\documentclass[11pt,sumlimits,nointlimits,namelimits,leqno]{amsart}

\usepackage{amsopn,amsfonts,amssymb,amscd,graphicx,tabularx,oldgerm,stmaryrd}
\usepackage{amsmath,amssymb}
\usepackage{graphicx}
\usepackage[mathscr]{eucal}
\usepackage[dvips]{color}
\usepackage[all]{xy}

\newcounter{thmnum}

\newtheorem{thm}[thmnum]{Theorem}
\newtheorem{lem}[thmnum]{Lemma}

\newcommand{\Gal}{\textup{Gal}}

\newcommand{\Disc}{\textup{Disc}}

\renewcommand{\Re}{\textup{Re} \,}
\renewcommand{\Im}{\textup{Im} \,}

\newcommand{\pf}{\noindent{\it Proof:} \,}

\newcommand{\rki}{\noindent{\bf Remark 1.} \,}
\newcommand{\rkii}{\noindent{\bf Remark 2.} \,}

\begin{document}

\title{On the density of abelian $\ell$-extensions}

\author{Chih-Yun Chuang}
\address{Department of Mathematics, National Taiwn University, Taipei City, Taiwan 10617}
\email{cychuang@ntu.edu.tw}

\author{Yen-Liang Kuan}
\address{Department of Mathematics, National Taiwan University, Taipei City, Taiwan 10617}
\email{ylkuan@ntu.edu.tw}



\begin{abstract}
We derive an asymptotic formula which counts the number of abelian extensions of prime degrees over rational function fields. Specifically, let $\ell$ be a rational prime and $K$ a rational function field $\Bbb F_q(t)$ with $\ell \nmid q$. Let $\Disc_f\left(F/K\right)$ denote the finite discriminant of $F$ over $K$. Denote the number of abelian $\ell$-extensions $F/K$ with $\deg\left(\Disc_f(F/K)\right) = (\ell-1)\alpha n$ by $a_{\ell}(n)$, where $\alpha=\alpha(q, \ell)$ is the order of $q$ in the multiplicative group $\left(\Bbb Z/\ell \Bbb Z\right)^\times$. We give a explicit asymptotic formula for $a_\ell(n)$. In the case of cubic extensions with $q\equiv 2 \pmod 3$, our formula gives an exact analogue of Cohn's classical formula.
\end{abstract}

\maketitle

\section{Introduction}

In arithmetic statistics, the problem of counting number fields is of particularly interest. Let $N_n(X)$ denote the number of isomorphism classes of number fields of degree $n$ over $\Bbb Q$ having absolute discriminant at most $X$. It is conjectured that $N_n(X)/X$ tends to a finite limit as $X$ tends to infinity and it is positive for $n > 1$. This conjecture is trivial for $n = 1$ and it is well-known for $n=2$. For the degree $n = 3$ it is a theorem of Davenport and Heilbronn \cite{DH}. In the past decade, Bhargava \cite{B1, B2} proved this conjecture for $n = 4$ and $5$.

Prompted by the work of Davenport and Heilbronn \cite{DH} on the density of cubic fields. Cohn showed that the cyclic cubic fields are rare compared to all cubic fields over $\Bbb Q$. More precisely, let $G$ be a fixed finite abelian group of order $m$ and let $F$ range over all abelian number fields with Galois group $\Gal\left(F/\Bbb Q\right)=G$. Denote by $\Disc\left(F/\Bbb Q\right)$ the absolute value of the discriminant of $F$ over $\Bbb Q$ and define $N_G(X)$ to be the number of abelian number fields $F$ with $\Disc\left(F/\Bbb Q\right) \leq X$. When $G= \Bbb Z/3\Bbb Z$, Cohn \cite{Coh} proved that
\begin{equation}\label{Cohn}
N_{G}(X) \sim \frac{11\pi}{72\sqrt{3}\zeta(2)}\prod_{p \equiv 1 \left(6\right)}\frac{\left(p+2\right)\left(p-1\right)}{p\left(p+1\right)}\sqrt{X}, \ \ \ \hbox{as $X \rightarrow \infty$,}
\end{equation}
where the product is taking over all primes $p \equiv 1 \pmod 6$. For arbitrary given finite abelian groups $G$, an asymptotic formula of $N_G(X)$ has been worked out by M\"{a}ki \cite{Mak}. In particular, fix $\ell$ a prime and $G=\Bbb Z/\ell\Bbb Z$, there is a constant $c>0$ such that (cf. Wright \cite[Theorem I.2]{Wri})
\begin{equation}\label{p}
N_{G}(X) \sim c \cdot X^{\frac{1}{\ell-1}}, \ \ \ \hbox{as $X \rightarrow \infty$}.
\end{equation}

In this paper, we prove very precise versions of the asymptotic formulas (\ref{Cohn}), (\ref{p}) for function fields. Let $K$ be a rational function field $\Bbb F_q(t)$ over the finite field $\Bbb F_q$. Fix a rational prime $\ell \nmid q$. Let $F$ be an abelian $\ell$-extension of $K$ and denote by $\mathcal{O}_F$ the integral closure of $A=\Bbb F_q[t]$ in $F$. Let $\Disc_f\left(F/K\right)$ denote the finite discriminant of $F$ over $K$ which means the discriminant of $\mathcal{O}_F$ over $A$. Note that $\Disc_f\left(F/K\right)$ is a $(\ell-1)$-th power for some square-free polynomial in $A$ since all ramified primes are totally tamely ramified in $F/K$. Denote $\alpha=\alpha(q, \ell)$ to be the order of $q$ in the multiplicative group $\left(\Bbb Z/\ell \Bbb Z\right)^\times$. Then the degree of ramified primes in $F$ must be divided by $\alpha$. Adapting an idea of Cohn, we shall study the partial Euler product
$$
f(s) := \prod_{\alpha \mid \deg(P)} \left(1 + (\ell - 1)q^{-\deg(P)s}\right),
$$
where the product is taking over all finite primes $P$ of $K$ such that $\alpha\mid\deg(P)$.

We will prove that $f(s)$ converges absolutely for $\Re s > 1$ and is analytic in the region $\{s \in B : \Re s = 1\}$ except for a pole of order $w:=(\ell - 1)/\alpha$ at $s = 1$ where $B := \{s\in\Bbb C : -\pi /\log q^{\alpha} \leq \Im s < \pi /\log q^{\alpha}\}$. The last statement implies that
$$
r(K, \ell) := \lim_{s\rightarrow 1} (s-1)^wf(s) > 0.
$$
Then we have the following main result

\begin{thm}\label{prime}
Let $\ell$ be a rational prime and $K=\Bbb F_q(t)$ with $\ell \nmid q$. Write $\alpha=\alpha(q, \ell)$ for the order of $q$ in the multiplicative group $\left(\Bbb Z/\ell \Bbb Z\right)^\times$ and set $w=(\ell - 1)/\alpha$. Denote the number of abelian $\ell$-extensions $F/K$ with $\deg\left(\Disc_f(F/K)\right) = (\ell-1)\alpha n$ by $a_{\ell}(n)$. Then, for any $\epsilon > 0$, we have
$$
a_{\ell}(n) = \begin{cases}
2 r(K, \ell)\log q \cdot q^{\alpha n} + O(q^{(\frac{\alpha}{2}+\epsilon) n}) \ \ \ &\hbox{if $w=1$,} \\
\frac{2 r(K,\ell)\log^w q^\alpha}{(\ell-1)(w-1)!}q^{\alpha n}P(n) + O(q^{(\frac{\alpha}{2}+\epsilon) n}) \ \ \ &\hbox{if $w > 1$,}
\end{cases}
$$
where $r(K,\ell)$ is defined as above and $P(X) \in \Bbb R[X]$ is a monic polynomial of degree $w-1$.
\end{thm}

\rki Note that, as in formula (\ref{p}), the growth order of the number $a_\ell(n)$ is a $(\ell-1)$-root of the number of polynomials of degree $(\ell-1)\alpha n$ when $w=1$.

\rkii By our much sharper Wiener-Ikehara Tauberian theorem (see Appendix for a function field version), the constants $r(K, \ell)$ and $P(X)$ can be explicitly determined (see Section 2 for details).

In the case of cyclic cubic fields, we derive explicit formulas for $r(K, \ell)$ and $P(X)$ as the following

\begin{thm}\label{cubic}
Let $K = \Bbb F_q(t)$ with $3\nmid q$ and $a_3(n)$ the number of cyclic cubic extensions $F/K$ with $\deg\left(\Disc_f(F/K)\right) = 2\alpha n$ where $\alpha = 1$ if $q \equiv 1 \pmod 3$ or $2$ otherwise. Let $P$ denote the finite primes of $K$ and $\zeta_A(s)$ is the zeta function of $A$.
Then, for any $\epsilon>0$, we have

\noindent\textup{(i)} If $q\equiv 1 \pmod 3$, let
$$
g(s) = \prod_{P}\left(1-3q^{-2\deg P s}+2q^{-3\deg P s}\right) \ \ \  \hbox{for $\Re s > 1/2$},
$$
where the product is taking over all finite primes $P$ of $K$. Then one has
$$
a_{3}(n) = g(1)q^n\left(n + 1 + \frac{g^\prime(1)}{g(1) \log q} \right)+O(q^{(\frac{1}{2}+\epsilon) n}) \ \ \ \hbox{as $n \rightarrow \infty$}.
$$

\noindent\textup{(ii)} If $q\equiv 2 \pmod 3$, one has
$$
a_{3}(n) = \frac{1}{\zeta_{A}(2)}\left[\prod\limits_{\deg(P) : even}\frac{\left(q^{\deg(P)} + 2\right)\left(q^{\deg(P)} - 1\right)}{q^{\deg(P)}\left(q^{\deg(P)} + 1\right)}\right]q^{2n} + O(q^{(1+\epsilon) n})) \ \ \ \hbox{as $n\rightarrow \infty$}.
$$
\end{thm}

Note that the asymptotic formula (ii) of Theorem \ref{cubic} is a function filed analogue of Cohn's result. This occurs because, when $q \equiv 2 \pmod 3$, the infinite prime $\infty = 1/t$ is unramified in cyclic cubic fields which happen to be the same situation as for the number field $\Bbb Q$. We are also able to treat the case $q\equiv 1 \pmod 3$ in (i) of Theorem \ref{cubic} with a bigger main term and exact second order term.

We now briefly describe the contents of this paper. In the next section, we use a version of Tauberian theorem (Theorem \ref{Tau}) to prove Theorem \ref{prime}. In Section 3, we adapt an idea of Cohn to prove Theorem \ref{cubic}. In the appendix, we derive the needed function field version of Wiener-Ikehara Tauberian theorem.

\section{Counting Abelian $\ell$-extensions}

We fix the following notations in this paper:
$$
\aligned
\ell &: \hbox{a fixed rational prime.} \\
K &: \hbox{the rational function field $\Bbb F_q(t)$ of the characteristic $p\neq \ell$.} \\
\alpha=\alpha(q,\ell) &: \hbox{the order of $q$ in the multiplicative group $\left(\Bbb Z/\ell\Bbb Z\right)^{\times}$.}\\
A &: \hbox{the polynomial ring $\Bbb F_q[t]$.} \\
\infty &: \hbox{the infinite place $1/t$.} \\
\mathcal{O}_\infty &: \hbox{the valuation ring of $\infty$.} \\
P &: \hbox{a finie prime (place) of $K$.} \\
K_P &: \hbox{the completion of $K$ at $P$.} \\
\mathcal{O}_P^{\times}&: \hbox{the group of $P$-units in $K_P$.} \\
\Bbb A_K^{\times} &: \hbox{the id\`{e}le group of $K$.} \\
F &: \hbox{an abelian $\ell$-extension over $K$.} \\
\Disc_f\left(F/K\right) &: \hbox{the finite discriminant of $F$ over $K$ identifies as a polynomial in $A$.} \\
\endaligned
$$

Note that we always have $\Disc_f\left(F/K\right) = D^{\ell - 1}$ for some square-free polynomial $D=D(F) \in A$ since all ramified primes are totally tamely ramified in $F/K$. From class field theory, we know that each abelian $\ell$-extension $F$ over $K$ corresponds exactly to $\ell - 1$ surjective homomorphisms
$$
\phi : \Bbb A_K^{\times}/K^{\times} \rightarrow \Gal\left(F/K\right) \cong \Bbb Z/\ell\Bbb Z.
$$
Recall that the image of $\mathcal{O}_P^\times$ is equal to the inertia group of $P$ in $\Gal\left(F/K\right)$ by local class field theory. Let $K_{\infty}^+ = \{1/t\} \times (1+1/t\mathcal{O}_\infty)$ be the direct product of the cyclic group $\{1/t\}$ generated by $1/t$ and the pro-$p$ group $(1+1/t\mathcal{O}_\infty)$. Since $\Bbb A_K^\times$ can be written as a direct product of $K^\times$ and $\left(\prod_{P}\mathcal{O}_P^\times \times K_\infty^+\right)$, the number of abelian $\ell$-extensions $F$ over $K$ with $\Disc_f\left(F/K\right) = D^{\ell -1}$ is equal to
$$
\frac{2}{\ell -1} \cdot \#\{\phi : \prod_{P\mid D} \mathcal{O}_P^\times \rightarrow \Bbb Z/\ell\Bbb Z : \hbox{the restriction of $\phi$ to $\mathcal{O}_P^\times$ is surjective for all $P\mid D$}\}.
$$
Since the kernel of the canonical map $\mathcal{O}_P^\times \rightarrow \left(\mathcal{O}_P/P\right)^\times$ is a pro-$p$ group and $p \neq \ell$, a map $\mathcal{O}_P^\times \rightarrow \Bbb Z/\ell\Bbb Z$ must factor through $\left(\mathcal{O}_P/P\right)^\times \rightarrow \Bbb Z/\ell\Bbb Z$. As the restriction of $\phi$ from $ \prod_{P\mid D} \mathcal{O}_P^\times$ to $\mathcal{O}_P^\times$ is surjective, we note that $\alpha(q,\ell)$ must divide $\deg(P)$. Combining these discussions, we conclude that the number of abelian $\ell$-extensions $F$ over $K$ with $D(F)=P_1\cdots P_m$ and $\alpha(q, \ell) \mid \deg(P_i)$ for all $i$ is equal to $\left(\ell - 1\right)^{m-1}$.

We are interested in the partial Euler product
$$
f(s) := \prod_{\alpha \mid \deg(P)} \left(1 + (\ell - 1)q^{-\deg(P)s}\right)\ \ \ \hbox{for $\Re s > 1$}.
$$
Write $f(s) = \sum b_{\alpha n} q^{-\alpha ns}$, then $a_{\ell}(n) = 2 b_{\alpha n}/(\ell -1)$ for all $n > 0$, where $a_{\ell}(n)$ is the number of abelian $\ell$-extensions $F$ over $K$ with $\deg\left(\Disc_f(F/K)\right) = (\ell - 1)\alpha n$. Let $\zeta_A(s)$ be the zeta function of $A$ which is defined by
$$
\aligned
\zeta_A(s) :&= \prod_{P}\left(1 - q^{-\deg(P) s}\right)^{-1}  \ \ \ \hbox{for $\Re s > 1$} \\
&= \frac{1}{1 - q^{1-s}}. \\
\endaligned
$$
We have

\begin{lem}\label{f}
Let $\zeta_{A_\alpha}(s)$ be the zeta function of $A_\alpha=\Bbb F_{q^\alpha}[t]$. Then
$$
f(s) = \zeta_{A_\alpha}^{\frac{\ell - 1}{\alpha}}(s)\cdot g(s) \ \ \ \hbox{for $\Re s > 1$},
$$
where $g(s)$ is an analytic function for $\Re s \geq 1$ and non-vanishing at $s = 1$.
\end{lem}

\pf For each $d\mid\alpha$, we consider the infinite product
$$
\mathcal{L}_{d}(s):= \prod_{P:\left(\alpha, \deg(P)\right)=d}\left(1-q^{-\frac{\alpha\deg(P)}{d}s}\right)^{-d} \ \ \ \hbox{for $\Re s > \frac{d}{\alpha}$},
$$
where the product is taking over the finite primes $P$ in $K$ such that $\gcd\left(\alpha, \deg(P)\right)=d$. Then we know that
$$
\zeta_{A_\alpha}(s) = \prod_{d\mid \alpha}\mathcal{L}_{d}(s) = \prod_{\alpha\mid\deg(P)}\left(1-q^{-\deg(P)s}\right)^{-\alpha}\prod_{\substack{d\mid \alpha\\ d\neq\alpha}}\mathcal{L}_{d}(s) \ \ \ \hbox{for $\Re s > 1$}.
$$
Let $w=(\ell-1)/\alpha$, then
$$
f(s) = \zeta_{A_\alpha}^{w}(s)\prod_{\alpha\mid\deg(P)}\left[\left(1 + (\ell - 1)q^{-\deg(P)s}\right)\left(1-q^{-\deg(P)s}\right)^{\ell-1}\right]\prod_{\substack{d\mid \alpha\\ d\neq\alpha}}\mathcal{L}_{d}^{-w}(s) \ \ \ \hbox{for $\Re s > 1$}.
$$
We know that $\mathcal{L}_d^{-1}(s)$ is analytic and non-vanishing at $s=1$ since $\mathcal{L}_d^{-1}(1) \geq \zeta_{A_\alpha}^{-1}(2)>0$ for all $d \neq \alpha$ (cf. \cite[Section 2]{CKY}). On the other hand, one computes
$$
\aligned
h(s)&:=\prod_{\alpha\mid\deg(P)}\left[\left(1 + (\ell - 1)q^{-\deg(P)s}\right)\left(1-q^{-\deg(P)s}\right)^{\ell-1}\right]\\
&=\prod_{\alpha\mid\deg(P)}\left(1+c_2q^{-2\deg(P)s}+\cdots+c_{\ell}q^{-(\ell)\deg(P)s}\right) \ \ \ \hbox{for $\Re s > 1$,}\\
\endaligned
$$
where
$$
c_i = \begin{cases}
(-1)^i\binom{\ell-1}{i}+ (-1)^{i-1} (\ell - 1)\binom{\ell-i}{i-1} \ \ \ \hbox{if $i < \ell$,} \\
(-1)^{\ell - 1}(\ell - 1) \ \ \ \hbox{if $i = \ell$.}
\end{cases}
$$
It is easy to check that $h(s)$ is analytic for $\Re s > 1/2$ and non-vanishing at $s=1$. This completes the proof of the lemma.\qed

Before we prove the main theorem, we recall a function field version of Wiener-Ikehara Tauberian theorem which we will prove in the appendix.

\begin{thm}\label{Tau}
Let $f(u) = \sum_{n \geq 0} b_nu^n$ with $b_n \in \Bbb C$ be convergent in the region $\{u\in\Bbb C : |u| < q^{-a}\}$. Assume that in the domain of convergence $f(u) = g(u)(u - q^{-a})^{-w} + h(u)$ holds, where $w \in \Bbb N$ and $h(u)$, $g(u)$ are analytic functions in $\{u\in \mathbb{C}: |u|< q^{-a+\delta}\}$ for some $\delta>0$,  $g(q^{-a})\not=0$. Then, for any $\epsilon > 0$, we have
$$
b_n = q^{an}Q(n) + O(q^{(a-\delta+\epsilon)n}),
$$
where $Q(X) \in \Bbb C[X]$ is the polynomial of degree $w-1$ given by
$$Q(X):=\displaystyle\sum_{j=0}^{w-2}\frac{g^{(j)}(q^{-a})(-1)^{w-j-1}}{j!(w-j-1)!}\left[\displaystyle\prod_{\ell=1}^{w-j-1}(X+\ell)\right]\cdot q^{a(w-j)}+\frac{g^{(w-1)}(q^{-a})}{(w-1)!}q^a.$$
In particular, if we write $Q(X)=\sum_{i=1}^{w-1}c_iX^{w-i}$, then we have
$$
c_1=(-1)^{-w}\frac{g(q^{-a})q^{aw}}{\Gamma(w)} \ \ \ \hbox{and} \ \ \ c_2=\frac{(-1)^{-w}}{\Gamma(w-1)}\cdot q^{aw} \cdot\left[g(q^{-a})(\frac{w}{2})-\frac{g'(q^{-a})}{q^a}\right].
$$
\end{thm}

We are now ready to prove Theorem \ref{prime}. Let $w=(\ell-1)/\alpha$ and define the region $B := \{s\in\Bbb C : -\pi /\log q^{\alpha} \leq \Im s < \pi /\log q^{\alpha}\}$. From Lemma \ref{f}, we have proved that $f(s) = \zeta_{A_\alpha}^{\frac{\ell - 1}{\alpha}}(s)\cdot g(s)$ for $\Re s > 1$,
where
$$
g(s)=\prod_{\alpha\mid\deg(P)}\left[\left(1 + (\ell - 1)q^{-\deg(P)s}\right)\left(1-q^{-\deg(P)s}\right)^{\ell-1}\right]\prod_{\substack{d\mid \alpha\\ d\neq\alpha}}\mathcal{L}_{d}^{-w}(s)
$$
is an analytic function for $\Re s \geq 1/2$ and non-vanishing at $s = 1$. Note that $f(s)$ converges absolutely for $\Re s > 1$ and is analytic in the region $\{s \in B : \Re s = 1\}$ except for a pole of order $w$ at $s = 1$. The last statement implies that
$$
r(K, \ell) := \lim_{s\rightarrow 1} (s-1)^wf(s)=\frac{g(1)}{\alpha\log q}> 0.
$$

Replacing the variable $s$ by $-\frac{\log u}{\log q^\alpha}$, define the function $Z_f(u) := f\left(-\frac{\log u}{\log q^\alpha}\right)=\sum b_{\alpha n}u^n$. Then we have $Z_f(u) = Z_g(u)\left(u -q^{-\alpha}\right)^{-w}$ where $Z_g(u)$ is an analytic function in the region $\{u\in\Bbb C : |u|< q^{-\alpha/2}\}$, and
$$
\aligned
Z_g\left(q^{-\alpha}\right) &= \lim_{u \rightarrow q^{-\alpha}} \left(u - q^{-\alpha}\right)^w Z_f(u) \\
&= \lim_{s \rightarrow 1} \frac{\left(q^{-\alpha s} - q^{-\alpha}\right)^w}{\left(s-1\right)^w}\left(s-1\right)^wf(s) \\
&= \left(-\frac{\log q^\alpha}{q^\alpha}\right)^wr(K,\ell). \\
\endaligned
$$
Applying Theorem \ref{Tau}, for any $\epsilon > 0$, as $n \rightarrow \infty$,
\begin{equation}\label{b_n}
b_{\alpha n}= q^{\alpha n}Q(n) + O(q^{(\frac{\alpha}{2}+\epsilon) n}), \\
\end{equation}
where $Q(X) \in \Bbb R[X]$ is a polynomial of degree $w-1$ given by
$$Q(X)=\displaystyle\sum_{j=0}^{w-2}\frac{Z_g^{(j)}(q^{-a})(-1)^{w-j-1}}{j!(w-j-1)!}\left[\displaystyle\prod_{\ell=1}^{w-j-1}(X+\ell)\right]\cdot q^{a(w-j)}+\frac{Z_g^{(w-1)}(q^{-a})}{(w-1)!}q^a.$$
Note that the leading coefficient $c_1$ of $Q(X)$ is given by
\begin{equation}\label{c1}
c_1=\frac{r(K,\ell)(\log q^\alpha)^w}{(w-1)!}.
\end{equation}
Combining (\ref{b_n}), (\ref{c1}) and $a_{\ell}(n) = 2 b_{\alpha n}/(\ell-1)$, we complete the proof of Theorem \ref{prime}.

\section{Counting cyclic cubic extensions}

In this section, we derive an explicit formula of $r(K, \ell)$ for the case of cyclic cubic extensions. Moreover, we give the second order term for the number $a_{3}(n)$ when $q \equiv 1 \pmod 3$. In Lemma \ref{f}, we have proved that $f(s)$ can be written as a product of the zeta function of $A_\alpha$ and an analytic function $g(s)$. We will write down $g(s)$ explicitly and compute its derivative to obtain Theorem \ref{cubic}. In the general case, it is also possible to compute $r(K, \ell)$ and small order terms for $a_{\ell}(n)$ using the same method.

In the case of $q \equiv 1 \pmod 3$, we consider the partial Euler product
$$
f(s) := \prod_{P} \left(1 + 2q^{-\deg(P) s} \right) \ \ \ \hbox{for $\Re s > 1$}.
$$
Write $f(s) = \sum b_{n} q^{-ns}$, then $a_{3}(n) = b_{n}$ for all $n > 0$. Recall that $\zeta_A(s)$ is the zeta function of $A$ which is defined by
$$
\zeta_A(s) := \prod_{P}\left(1 - q^{-\deg(P) s}\right)^{-1} \ \ \ \hbox{for $\Re s > 1$}.
$$
Then one computes that $f(s) = \zeta_A^2(s)g(s)$ where
$$
g(s) = \prod_{P}\left(1 - 3q^{-2\deg(P) s} + 2q^{-3\deg(P) s}\right) \ \ \ \hbox{for $\Re s > 1/2$}
$$
is analytic for $\Re s > 1/2$ and non-vanishing at $s=1$.

Replacing the variable $s$ by $-\frac{\log u}{\log q}$, define the function $Z_f(u) := f\left(-\frac{\log u}{\log q}\right)=\sum b_nu^n$. Then $Z_f(u) = Z_g(u)\left(u - q^{-1}\right)^{-2}$ where
$$
Z_g(u) = q^{-2}\prod_{P} \left(1-3u^{2\deg(P)} + 2u^{3\deg(P)}\right)
$$
is an analytic function in the region $\{u\in\Bbb C : |u|\leq q^{-1}\}$. Applying Theorem \ref{Tau} to $Z_f(u)$ and $a_{3}(n) = b_n$, there is a constant $\delta < 1$ such that
$$
a_{3}(n) = q^{2n}\left(q^2Z_g(q^-1)n+q^2Z_g(q^{-1})-qZ_g^{\prime}(q^{-1})\right)+O\left(q^{2n\delta}\right) \ \ \ \hbox{as $n \rightarrow \infty$}.
$$

In the case of $q\equiv 2 \pmod 3$, we consider the partial Euler product
$$
f(s) := \prod_{2\mid\deg(P)} \left(1 + 2q^{-\deg(P) s} \right) \ \ \ \hbox{for $\Re s > 1$}.
$$
We now write $f(s)=\sum b_{2n}q^{-2ns}$, then $a_{3}(n)=b_{2n}$ for all $n > 0$. Let $\zeta_{A_2}(s)$ be the zeta function of $A_2=\Bbb F_{q^2}[t]$. One has
$$
\zeta_{A_2}(s) = \prod_{\deg(P) : even} \left(1 - q^{-\deg(P) s}\right)^{-2}\prod_{\deg(Q) : odd}\left(1 - q^{-2\deg(Q) s}\right)^{-1}
$$
where $P$ are the primes in $A$ of odd degree and $Q$ are the primes in $A$ of odd degree. Then we have the identity
$$
f(s) = \zeta_{A_2}(s)\zeta_{A}^{-1}(2s) \prod_{\deg(P) : even}\frac{\left(q^{\deg(P) s} + 2\right)\left(q^{\deg(P)s} - 1\right)}{q^{\deg(P)s}(q^{\deg(P) s}+1)} \ \ \ \hbox{for $\Re s > 1$}.
$$

We now replace the variable $s$ by $-\frac{\log u}{\log q^2}$ and define the function $Z_f(u) := f\left(-\frac{\log u}{\log q^2}\right)=\sum b_{2n}u^n$. Then $Z_f(u) = Z_g(u)\left(u - q^{-2}\right)^{-1}$ where
$$
Z_g(u) = -q^{-2}(1-qu)\prod_{\deg(P) : even}\frac{\left(u^{\frac{-\deg(P)s}{2}}+2\right)\left(u^{\frac{-\deg(P)s}{2}}-1\right)}{u^{\frac{-\deg(P)s}{2}}\left(u^{\frac{-\deg(P)s}{2}} - 1\right)}
$$
is an analytic function in the region $\{u\in\Bbb C : |u|\leq q^{-2}\}$. From Theorem \ref{Tau} and $a_{3}(n) = b_{2n}$, there is a $\delta < 1$ such that
$$
a_{3}(n) = \frac{1}{\zeta_{A}(2)}\left[\prod\limits_{\deg(P) : even}\frac{\left(q^{\deg(P)} + 2\right)\left(q^{\deg(P)} - 1\right)}{q^{\deg(P)}\left(q^{\deg(P)} + 1\right)}\right]q^{2n} + O\left(q^{2n\delta}\right) \ \ \ \hbox{as $n\rightarrow \infty$}.
$$
This completes the proof of Theorem \ref{cubic}.

\section*{Appendix}

In the appendix, we derive a function field version of Wiener-Ikehara Tauberian Theorem with a main term much sharper then the standard one \cite[Corollary of Theorem 17.4]{Ros}. Before we prove the main theorem, we first record some lemmas.
\begin{lem}
Let $\Gamma(t)$ be  Gamma function. Then we have
$$\frac{n! n^t}{\displaystyle\prod_{i=0}^n (t+i)}=\Gamma(t)\left[1-\left(\frac{t^2+t}{2}\right)\frac{1}{n}+O(\frac{1}{n^2})\right] \mbox{ as }n\rightarrow \infty$$
for all complex numbers $t$, except the non-positive integers.
\end{lem}
\begin{proof}
Just use Stirling's formula.
\end{proof}

\begin{lem}\label{lem-T-1}
Let  $\delta, a>0$ be fixed real numbers and $q\geq2$. Assume  $w<1$  and $w  \in \mathbb{R}-\mathbb{Z}.$  Then we have
$$\displaystyle\int_{q^{-a}}^{q^{-a(1-\delta)}}\frac{du}{u^{n+1}\cdot (u-q^{-a})^w}=\frac{2\pi i\cdot e^{-w\pi i}}{1-e^{-2\pi iw}}\cdot \frac{1}{\Gamma(w)}\cdot \left[ 1+\left(\frac{w^2-w}{2}\right)\frac{1}{n}+O(\frac{1}{n^2})\right]q^{a(n+w)}n^{w-1} \ \ \ \hbox{as $n\rightarrow \infty $.}
$$
Here  $(u-q^{-a})^w:=e^{w \log (u-q^{-a})}$, analytic branch of $\log u$ is  fixed with  $0<\mbox{\rm Arg }u<2\pi$. Thus $\log (u-q^{-a})$ is an analytic function defined on $\mathbb{C}-[q^{-a},\infty)$.
\end{lem}

\begin{proof}
Substituting $q^{-a}u$ for $u$, the integral changes to
$$
\aligned
\displaystyle q^{a(n+w)}\cdot \displaystyle \int_1^{q^{a\delta}} \frac{du}{u^{n+1}\cdot (u-1)^w}
&=q^{a(n+w)}\cdot \left[ \displaystyle \int_0^{\infty} \frac{du}{(u+1)^{n+1}\cdot u^w}-\displaystyle \int_{q^{a\delta}-1}^\infty \frac{du}{(u+1)^{n+1}\cdot u^w}  \right] \\
&:= q ^{a(n+w)}(I_1+I_2).\\
\endaligned
$$
The integral $I_2$  simply equals  to $O\left( \frac{q^{-a\delta n}}{n}\right).$ The integral $I_1$ can be checked to be
$$\frac{2\pi i}{1-e^{-2\pi i w}}\cdot \mbox{Res}\left( \frac{u^{-w}}{(u+1)^{n+1}},-1\right).$$
Furthermore, we have
\begin{align}
\mbox{Res}\left( \frac{u^{-w}}{(u+1)^{n+1}},-1\right)\notag
&=\frac{1}{n!\cdot n^w}\left( \displaystyle\prod_{i=0}^{n-1}(-w-i)(w+n)\cdot(-1)^{-w-n}\right)\cdot \frac{n^w}{n+w} \notag\\
&=\frac{e^{-\pi iw}}{\Gamma\left( w\right)}\left[ 1+\left(\frac{w^2-w}{2}\right)\frac{1}{n}+O(\frac{1}{n^2})\right]n^{w-1}.\notag
\end{align}
This gives the desired identity.

\end{proof}

Let $q\geq 2$, our Tauberian Theorem is

\begin{thm}\label{thm-T}
Let $f(u):=\displaystyle\sum_{n\geq 0} a_n u^n$ with the  numbers $a_n\in \mathbb{C}$ for all $n$, be convergent in $$\{u\in \mathbb{C}: |u|<q^{-a}\}$$ for a fixed real number $a>0$. Assume that in the above domain  $$f(u)=g(u)(u-q^{-a})^{-w}+h(u)$$ holds, where $w>0$ and $h(u)$, $g(u)$ are analytic functions in $\{u\in \mathbb{C}: |u|< q^{-a+\delta}\}$ for some $\delta>0$. Then
$$
a_n \sim (-1)^{-w}\frac{g(q^{-a})q^{aw}}{\Gamma(w)}\cdot q^{an}n^{w-1} \ \ \ \hbox{as $n \rightarrow \infty$}.
$$
Moreover, if $w$ is a positive integer, then, for any $\epsilon > 0$, one has
\begin{equation}\label{17.4}
a_n = Q(n)\cdot q^{an} + O(q^{(a-\delta+\epsilon)n}), \mbox{ as }n\rightarrow \infty,
\end{equation}
where $Q(X) \in \Bbb C[X]$ is the polynomial of degree $w-1$ given by
$$Q(X):=\displaystyle\sum_{j=0}^{w-2}\frac{g^{(j)}(q^{-a})(-1)^{w-j-1}}{j!(w-j-1)!}\left[\displaystyle\prod_{\ell=1}^{w-j-1}(X+\ell)\right]\cdot q^{a(w-j)}+\frac{g^{(w-1)}(q^{-a})}{(w-1)!}q^a.$$
In particular, if we write $Q(X)=\sum_{i=1}^{w-1}c_iX^{w-i}$, then we have
$$
c_1=(-1)^{-w}\frac{g(q^{-a})q^{aw}}{\Gamma(w)} \ \ \ \hbox{and} \ \ \ c_2=\frac{(-1)^{-w}}{\Gamma(w-1)}\cdot q^{aw} \cdot\left[g(q^{-a})(\frac{w}{2})-\frac{g'(q^{-a})}{q^a}\right].
$$
When $0 < w < 1$, we also have
$$
a_n = c_1q^{an}n^{w-1} + c_2q^{an}n^{w-2} +O\left(q^{an}n^{w-3}\right) \ \ \ \hbox{as $n \rightarrow \infty$.}
$$
\end{thm}

The proof will be divided into three steps.\\
\rm{1}. If $w$ is a positive integer, then we adapt an idea of Rosen \cite[Theorem 17.4]{Ros}. But we remove the condition that the pole  lie at $u=q^{-1}$. In this way, we are able to obtain  the polynomial of $Q(n)$ precisely  for $w \geq 1$.\\
\rm{2}. If $0<w<1$, then we use a keyhole shape contour to carry out needed estimations.\\
\rm{3}. If $w>0$ is not a positive integer, then we proceed by induction on $\lfloor w\rfloor$, since we have proved the case $\lfloor w \rfloor
=0$ in step $\rm{2}$.

\bf{Step} \rm{1}:
\begin{proof}
  The case $w=1$ and the equation (\ref{17.4}) is covered in \cite[Theorem 17.1]{Ros} and \cite[Theorem 17.4]{Ros};
we give here exact formula for $c_i$ for $1\leq i \leq w$ when $w >1$.
We take a $0<\delta<1$ such that $g(u)$ and $h(u)$ are analytic on the disc : $\{s\in \mathbb{C}: |u|\leq q^{-a(1-\delta)}\}$. Let $C$ be the boundary of this disc oriented counterclockwise and $C_\epsilon$ a small circle about $u=0$ oriented clockwise. We have
$$\frac{1}{2\pi i}\displaystyle\oint_{C_\epsilon+C}\frac{f(u)}{u^{n+1}}du=-a_n+\frac{1}{2\pi i}\displaystyle\oint_{C}\frac{f(u)}{u^{n+1}}du.$$
Observe that
\begin{align}
\frac{1}{2\pi i}\displaystyle\oint_{C_\epsilon+C}\frac{f(u)}{u^{n+1}}du
=&\mbox{Res}\left(\frac{g(u)}{(u-q^{-a})^wu^{n+1}} ,q^{-a}\right)\notag\\
=&\mbox{Res}\left(\displaystyle\sum_{j=1}^\infty\frac{g^{(j)}(q^{-a})}{j!(u-q^{-a})^{w-j}u^{n+1}} ,q^{-a}\right)\notag\\
=&\frac{g^{(w-1)}(q^{-a})}{(w-1)!}q^{a(n+1)}+\displaystyle\sum_{j=0}^{w-2}\frac{g^{(j)}(q^{-a})(-1)^{w-j-1}}{j!(w-j-1)!}\left[\displaystyle\prod_{\ell=1}^{w-j-1}(n+\ell)\right]\cdot q^{a(n+w-j)}.
 \notag
\end{align}
and for any $\epsilon > 0$,
$$\frac{1}{2\pi i}\displaystyle\oint_{C}\frac{f(u)}{u^{n+1}}du\ll q^{(a-\delta+\epsilon)N}.$$
The proof is therefore complete for arbitrary positive integer $w>1$.
\end{proof}

{\bf{Step}} $\rm{2}$:

We use the contour which consists of a small circle $C_{0}$ with the center at the origin and a keyhole contour (cf. \cite[Section 3]{CKY} for more details). The keyhole contour consists of a small circle $C_{\epsilon}$ about the $q^{-a}$ of radius $\epsilon$, extending to a line segment $\gamma_1$  parallel and close to the positive real axis but not touching it, to an almost full circle $C_1$, returning to a line segment parallel $\gamma_2$, close, and below the positive real axis in the negative sense, returning to the small circle.
 Therefore,  for each $n \in\mathbb{N}$, one has
\begin{itemize}
\item[(a)]
$
a_{n} = \frac{1}{2 \pi i}\oint_{C_1 + \gamma_1 + C_{\epsilon} +{\gamma_2}} \frac{f(u)}{u^{n+1}} du.
$
\item[(b)]$
\int_{C_1}\frac{f(u)}{u^{n + 1}} du =   o\left(\frac{q^{an}}{ n^{2-w}}\right), \mbox{ as }n\rightarrow \infty.$
\item[(c)]$
\lim_{\epsilon \rightarrow 0^+}\displaystyle\int_{C_\epsilon} \frac{f(u)}{u^{n + 1}} du=0.
$
\item[(d)]For all $0<w<1$, we have
$$
\aligned
\lim_{\epsilon \rightarrow 0^+}\frac{1}{2\pi i}\cdot\int_{\gamma_1 + \gamma_2}\frac{f(u)}{u^{n + 1}} du =&\frac{g(q^{-a})(-a)^{1-w}q^{aw}}{\Gamma(w)}\left(\frac{q^{an}}{(an)^{1-w}}\right)\\
&-\frac{(-1)^{-w}}{\Gamma(w-1)}q^{aw}\notag \cdot \left[g(q^{-a})(\frac{w}{2})-\frac{g'(q^{-a})}{q^a} \right]\left(\frac{q^{an}}{n^{2-w}}\right)\notag\\ &+O(\frac{q^{an}}{n^{3-w}}),\mbox{ as }n\rightarrow \infty.\notag
\endaligned
$$
\end{itemize}
Combining (a)$\sim$(d),  we arrive at this theorem for case $0<w<1$.

\bf{Step} \rm{3}:
\begin{proof}
We proceed by induction on $\lfloor w\rfloor$.  The case $\lfloor w \rfloor=0$  has been proved. Assume that Theorem \ref{thm-T} holds for  all $\lfloor w\rfloor\leq N$.
Suppose $\lfloor w\rfloor=N+1$, we consider the following function
\begin{align}
\hat{f}(u):=&f(u)\cdot (u-q^{-a})\notag\\ =&g(u)\cdot (u-q^{-a})^{-(w-1)}+h(u)\cdot (u-q^{-a}), \notag
\end{align}
We now write $\hat{f}(u) = \sum_{n\geq0} b_nu^n$, where $b_n= -q^{-a}a_n+a_{n-1}$ for all $n\geq 0$ and set $a_{-1}=0$.
Note that $\hat{f}(u)$ is still a meromorphic  function on $$\{ u\in \mathbb{C}-[q^{-a},\infty): |u|<q^{-a(1-\delta')}\}$$
for some $\delta'>0$ and $\displaystyle\lim_{u\rightarrow q^{-a}}(u-q^{-a})^{w-1}\cdot \hat{f}(u)=g(q^{-a})\not=0$.
Solving $a_n$ from $b_n$,  we derive that
\begin{align}
a_n=-\displaystyle\sum_{k=0}^{n}q^{a(k+1)}\cdot b_{n-k} \mbox{ for }n\geq 0.\label{T-1}
\end{align}
Since $\displaystyle\lim_{r\rightarrow q^{-a}}(u-q^{-a})^{w-1} \cdot \hat{f}(u)\not=0$  and  $\lfloor w-1 \rfloor\leq N$, we have
$$
b_{n}=(-1)^{1-w}\cdot\frac{g(q^{-a})q^{a(w-1)}}{\Gamma(w-1)}\left(\frac{q^{an}}{n^{2-w}}\right)+o\left(\frac{q^{an}}{n^{2-w}}\right)
$$
by induction hypothesis. Combining (\ref{T-1}) and the above formula of $b_n$ implies that
$$a_n=(-1)^{-w}\frac{g(q^{-a}) q^{a(w-1)}}{\Gamma(w-1)} \cdot     \left(q^{a(n+1)}\right)\cdot \left(\displaystyle \sum_{k=1}^{n}k^{w-2} \right)+o\left(\frac{q^{an}}{n^{1-w}}\right).$$
Note that for $w> 1$, we have
$$\displaystyle \sum_{k=1}^{n}k^{w-2}=\left(\frac{n^{w-1}}{w-1}\right)+o(n^{w-1}).$$ Hence
\begin{align}
a_n=&(-1)^{-w}\frac{g(q^{-a}) q^{a(w-1)}}{\Gamma(w-1)} \cdot     \left(q^{a(n+1)}\right)\cdot \left[ \left(\frac{n^{w-1}}{w-1}\right)+o(n^{w-1}) \right]+o\left(\frac{q^{an}}{n^{1-w}}\right)
\notag\\
=&(-1)^{-w}\frac{g(q^{-a}) q^{aw}}{\Gamma(w)} \cdot \left(\frac{q^{an}}{n^{1-w}}\right)+o\left(\frac{q^{an}}{n^{1-w}}\right).
 \notag
\end{align}
The proof is complete.
\end{proof}

\end{document}